\documentclass[12pt,psamsfonts]{amsart}
\usepackage{amssymb,amsmath,amscd,eucal,epsfig}
\usepackage{tikz}

 \def\Dj{\hBox{D\kern-.73em\raise.30ex\hBox{-} \raise-.30ex\hBox{}}}
 \def\dj{\hBox{d\kern-.33em\raise.80ex\hBox{-} \raise-.80ex\hBox{\kern-.40em}}}

\def\<{\langle}                     
\def\>{\rangle}                     

\theoremstyle{plain}
\newtheorem{theorem}{Theorem}[section]

\newtheorem{lemma}{Lemma}[section]

\newtheorem{remark}{Remark}[section]

\theoremstyle{definition}
\newtheorem{definition}{Definition}[section]

\newtheorem{rem}{Remark}[section]

\numberwithin{equation}{section}

\begin{document}
\setcounter{page}{1}


\title[Radio Number of Stacked-Book Graphs]{On Radio Number of Stacked-Book Graphs}


\author[T.C Adefokun]{Tayo Charles Adefokun$^1$ }
\address{$^1$Department of Computer and Mathematical Sciences,
\newline \indent Crawford University,
\newline \indent Nigeria}
\email{tayoadefokun@crawforduniversity.edu.ng}

\author[D.O. Ajayi]{Deborah Olayide Ajayi$^2$}
\address{$^2$Department of Mathematics,
\newline \indent University of Ibadan,
\newline \indent Ibadan,
\newline \indent Nigeria}
\email{olayide.ajayi@mail.ui.edu.ng; adelaideajayi@yahoo.com}





\begin{abstract}
A Stacked-book graph $G_{m,n}$ results from the Cartesian product of a star graph $S_m$ and path $P_n$, where $m$ and $n$ are the orders of $S_m$ and $P_n$ respectively. A radio labeling problem of a simple and connected graph, $G$, involves a non-negative integer function $f:V(G)\rightarrow \mathbb Z^+$ on the vertex set $V(G)$ of G, such that for all $u,v \in V(G)$, $|f(u)-f(v)| \geq \textmd{diam}(G)+1-d(u,v)$, where $\textmd {diam}(G)$ is the diameter of $G$ and $d(u,v)$ is the shortest distance between $u$ and $v$. Suppose that $f_{min}$ and $f_{max}$ are the respective least and largest values of $f$ on $V(G)$, then, span$f$, the absolute difference of $f_{min}$ and $f_{max}$, is the span of $f$ while the radio number $rn(G)$ of $G$ is the least value of span$f$ over all the possible radio labels on $V(G)$. In this paper, we obtain the radio number for the stacked-book graph $G_{m,n}$ where $m \geq 4$ and $n$ is even, and obtain bounds for $m=3$ which improves existing upper and lower bounds for $G_{m,n}$ where $m=3$.
\end{abstract}


\maketitle


\section{Introduction}
The graph $G$ considered in this paper is simple and undirected. The vertex and edge sets of $G$ are $V(G)$ and $E(G)$. For $e=uv \in E(G)$, $e$ connects two vertices $u$ and $v$ while $d(u,v)$ is the distance between $u,v$ and $\textmd{diam}(G)$ is  the diameter of $G$. Radio number labeling problem, which is mostly applied in frequency assignment for signal transmission, where it mitigates the problems of signal interference. It was first suggested in 1980 by Hale\cite{Hale}.

Let $f$ be a non negative integer function on $V(G)$ such that the radio labeling condition, $|f(u)-f(v)| \geq \textmd{diam} {G}+1-d(u,v)$ is satisfied for every pair $u,v \in V(G)$. The span of $f$, span$f$, is the difference between $f_{min}$ and $f_{max}$, the minimum and the maximum radio label on $G$ respectively. Thus the smallest possible value of span$f$ is the radio number, $rn(G)$, of $G$. The radio labeling condition guarantees that every vertex on $G$ has unique radio label. Therefore, $rn(G) \geq |V(G)|-1$ is trivially true. However, establishing the radio number of graphs has proved to be quite tedious. Even so, such numbers have been completely determined for some graphs. Liu and Zhu \cite{LZ1} showed that  for path, $P_n$, $n \geq 3$,
\begin{center}

$rn(P_n) = \left\{
\begin{array}{ll}
            2k(k-1)+1 &  \mbox{if} \;\; n=2k;   \\
             2k^2+2 &  \mbox{if} \; \; n=2k+1.
					
\end{array}
\right.$

\end{center}
 This improves results in  \cite{CEHZ1} and \cite{CEZ1} by Chatrand, et. al. where the upper and lower bounds for the same class of graph are obtained. Furthermore, Liu and Xie, \cite{LX2}, found the radio number for the square of a path, $P^2_n$ as:

\begin{center}

$rn(P^2_n) = \left\{
\begin{array}{ll}
            k^2+2 &  \mbox{if} \;\; n \equiv 1 (\rm{mod} \; 4), n \geq 9;   \\
             k^2+1 &  \mbox{if} \; \; \rm{otherwise}.
					
\end{array}
\right.$

\end{center}
Similar results are obtained in \cite{LX1} for square of cycles. Jiang \cite{J1} completely solved the radio number problem for the grid graph $(P_m \Box P_n)$, where for $m,n > 2$, it is noted that $rn(P_m \Box P_n)=\frac{mn^2+nm^2-n}{2}-mn-m+2$, for $m$-odd and $n$ even.
Saha and Panigrahi \cite{SP1} and Ajayi and Adefokun  \cite{AA1} obtained results on the radio numbers of Cartesian products of two cycles (toroidal grid) and of path and star graph (stacked-book graph) respectively. In the case of stacked-book graph $G=S_n \Box P_m$, $rn(G) \leq n^2m+1$, which the authors noted is not tight. Recent results on radio number include those on middle graph of path \cite{BD1}, trees, \cite{BD2} and edge-joint graphs \cite{NSS1}.

In this paper, for even positive integer $n$, we consider the stacked-book graph $G_{m,n}$ and derive the $rn(G_{m,n})$ for the case $m \geq 4$. Furthermore, new lower and upper bounds of the number are obtained for $m=3$, which improve similar results in \cite{AA1}.


\section{Preliminaries}
Let $S_m$ be a star of order $m \geq 3$ and for each vertex $v_i \in V(S_m)$, $2 \leq i \leq m$, $v_i$ is adjacent to $v_1$, the center vertex of $S_m$. Also, let $P_n$ be a path such that $|V(P_n)|=n$. The Graph $G_{m,n}=S_m \Box P_n$, is obtained by the Cartesian product of $S_m$ and $P_n$. The vertex set $V(G_{m,n})$ is the Cartesian product $V(S_m) \times V(P_n)$, such that for any $u_iv_j \in V(G_{m,n})$, then, $u_i \in V(S_m)$ and $v_j \in E(P_n)$. For $E(G_{m,n})$, $u_iv_j \; u_kv_l$ is contained in $E(G_{m,n})$ for $u_iv_j$, $u_kv_l \in V(G_{m,n})$, then either $u_i=u_k$ and $v_jv_l \in E(P_m)$ or $u_iu_k \in E(S_m)$ and $v_j = v_l$. Geometrically, $V(G_{m,n})$ contains $n$ number of $S_m$ stars, namely $S_{m(1)}, S_{m(2)}, \cdots, S_{m(n)}$, such that for every pair $v_i \in S_{m(i)}$ and $v_{i+1} \in S_{m(i+1)}$, $v_iv_{i+1} \in E(G_{m,n})$. These are, in fact, the only type of edges on $G_{m,n}$ apart from those on its $S_m$ stars. This geometry fetched $G_{m,n}$ the name \emph{stacked-book} graph.
\begin{remark}\label{rem01}It is easy to see that $diam(G_{m,n})=n+1$, being the number of edges from $u_iv_1 \rightarrow u_1v_1 \rightarrow u_1v_2 \rightarrow \cdots \rightarrow u_1v_n \rightarrow u_jv_n$, where $i \neq j$.
\end{remark}
\begin{rem} For convenience, we write $u_iv_j$ as $u_{i,j}$ in certain cases and $u_{i,j}u_{k,l}$ is the edge induced by $u_iv_j$ and $u_kv_l$.
\end{rem}

\begin{definition}
Let $G_{m,n}=S_m \Box P_n$. The vertex set $V_{(i)} \subset V(G_{m,n}) $ is the set of vertices on star $S_{m(i)}$, defined by the set $\left\{u_1v_i, u_2v_i, \cdots, u_mv_i\right\}$.
\end{definition}

We introduce the following definition:
\begin{definition}
Let $G_{m,n}=S_m \Box P_n$. Then, the pair $\left\{S_{m(i)},S_{m(i+\frac{n}{2})}\right\}$ is a subgraph $G(i) \subseteq G_{m,n}$ induced by $V_{(i)}$ and $V_{(i+ \frac{n}{2})}$.
\end{definition}

\begin{remark}
The maximum number of $G(i)$ subgraph in a $G_{m,n}$ graph, $n$ even, is $\frac{n}{2}$ and the $diam(G(i))=\frac{n}{2}+2$.
\end{remark}

\begin{remark}\label{reme1}
Let $\left\{V_{(i)}, V_{(i+\frac{n}{2})}\right\}$ induce $G(i)$, such that $V_{(i)}=\left\{u_{1,i},u_{2,i}, \cdots, u_{m,i}\right\}$ and $V_{(i+\frac{n}{2})}=\left\{u_{1,\frac{i+n}{2}}, u_{2,\frac{i+n}{2}}, \cdots, u_{m,\frac{i+n}{2}} \right\}$. Then, for $u\in V_{(i)}$, $v\in V_{(i+\frac{n}{2})}$ and $d(u,v)=p$, where $p \in \left\{\frac{n}{2}, \frac{n}{2}+1, \frac{n}{2}+2 \right\}$ and for $u_{k,i}, v_{t, i+\frac{n}{2}}$,

\begin{center}$p = \left\{
\begin{array}{ll}
            \frac{n}{2} &  \mbox{if} \;\; k=t;   \\
            \frac{n}{2}+1 &  \mbox{if} \; \; t=1, k\neq t; \\
						\frac{n}{2}+2 &  \mbox{if} \; \; t\neq 1, k \neq 1, k\neq t.
\end{array}
\right.$
\end{center}

\end{remark}

 \section{results}
In this section, we estimate the radio number of stacked-book graphs and obtain the exact radio number for $G_{m,n}$, for $m \geq 4$, $n$ even.
\begin{lemma} \label{lem0}
Let $S_m$ be a star on $G_{m,n}$ and $f$, a radio label function on $G_{m,n}$. Then span$f$ on $S_m$ is $n(m-1)+1$.
\end{lemma}
\begin{proof}
Let the center vertex of $S_m$ be $v_1$ and let $f(v_1)$ be the radio label on $v_1$. There exists some $v_2 \in V(S_m)$ such that $d(v_1,v_2)=1$. Therefore, by the definition, $f(v_2) \geq f(v_1)+n+1$. Suppose that $k \notin \left\{1,2\right\}$. For $v_k$, $d(v_2,v_k)=2$, for all $v_k \in V(S_m)$. Thus, without loss of generality, suppose that $v_m$ is the last vertex on $V(S_m)$, then $f(v_m) \geq f(v_0)+(n+1)+n(m-2)$ and the claim follows.
\end{proof}
\begin{rem} \label{rem0} It is easy to confirm that given a star $S_m$ with center vertex $v_1$, if for a positive integer $\alpha$, $rn(S_m)=\alpha$, then either $f(v_1)$ is $f_{min}$ or $f_{max}$.
\end{rem}
 Now we establish lower bound for $G(i)$.

\begin{lemma} \label{lem01}
Let $G(i)$ be a subgraph of $G_{m,n}$ and let $f$ be the radio label on $V(G_{m,n})$. Then, $rn(G(i)) \geq f(v_1)+mn-\frac{n}{2}+2$, where $v_1$ is the center vertex of $S_{m(i+\frac{n}{2})}$.
\end{lemma}
\begin{proof}
Let $S_{m(i)}$ and $S_{m(j)}$ be the stars on $G(i) \subset G_{m,n}$, where $j=i+\frac{n}{2}$. By Lemma \ref{lem0}, $f(v_m)=f(v_1)+mn-n+1$, with $f(v_m)=\max\left\{f(v_t): v_t \in V(S_m(j)) \right\}$, and $v_1$ the center of $S_{m(j)}$. Now, let $u_1$ be the center vertex of $S_{m(i)}$. It is clear that $d(u_1,v_1)=\frac{n+2}{2}$. Thus, $f(u_1) \geq f(v_1)+mn-n+1+\frac{n+2}{2} = f(v_1)+mn-\frac{n}{2}+2$. {\bf Claim:} For optimal radio labeling of $G(i)$, maximum label on $S_{m(i)}$ is at least $f(u_1)$.
{\bf Reason:} Consider some $u_m \in V(S_m)$, such that $m \neq 1$ and $d(u_m,v_m)=\frac{n}{2}+2$. Then $f(u_m)=f(v_1)+mn-\frac{n}{2}+1$. By Lemma \ref{lem0}, the span$f$ of $f$ for a star $S_m$ is $mn-n+1$. Now, $f(u_m)-(mn-n+1)=f(v_i)+\frac{n}{2}$. Thus, by Remark \ref{lem0}, $f(u_1)=f(v_1+\frac{n}{2})$. This is a contradiction, considering that $d(u_1,v_1)=\frac{n}{2}$.

\end{proof}
\begin{lemma}
Let $G^+(i) \subset G_{m,n}$ be $G(i) \cup w_1$, where $w_1$ is the center vertex of $S_{m(j+1)}$ and let $f$ be a radio labeling on $G_{m,n}$, where $n$ is even. Then, the span$f$ of $f$ on $G^+(i) \geq mn+3$
\end{lemma}
\begin{proof}
Let $u_1$ be the center vertex of $S_{m(i)}$. It can be verified that $d(u_1,w_1)=\frac{n}{2}$. By the proof of Lemma \ref{lem01}, $f(u_1) \geq f(v_1)+mn-\frac{n}{2}+2$, where $v_1$ is the center vertex of $S_{m(j)}$. Thus by definition, $f(w_1) \geq f(v_1)+mn-\frac{n}{2}+2 + \frac{n+2}{2}=f(v_1)+mn+3$. Since $f(v_1)$ is the minimum label on $G(i)$, the result follows.
\end{proof}
Now we present the lower bound for stacked-book graph $G_{m,n}$, where $n$ is an even integer and $m \geq 3$.
\begin{theorem}\label{thm11}
Let $G=G_{m,n}$ be a stacked-book graph with $m \geq 3$ and $n$ an even integer. Furthermore, let $f$ be the radio labeling on $G$. Then, $rn(G) \geq \frac{mn^2}{2}+n-1$.
\end{theorem}
\begin{proof}
From the definition of $G(i)$, graph $G_{m,n}$ contains $\frac{n}{2}$ number of $G(i)$  subgraphs. Likewise, it can be seen that $G_{m,n}$ contains $\frac{n-1}{2}$ number of $G^+(i)$ subgraphs. Now, let $G(\frac{n}{2})$, induced by $S_{m(\frac{n}{2})}$ and $S_{m(n)}$ be the last $G(i)$ subgraphs on $G_{m,n}$ and $G^+(1), G^+(2), \cdots, G^+(\frac{n-1}{2})$ be the $\frac{n-1}{2}$ number of $G^+(i)$ graphs. By the earlier result, if $f(v_1)=0$, then $rn(G_{m,n}) \geq \left(\frac{n-1}{2}\right)\left(mn+3\right)+mn-\frac{n}{2}+2=\frac{mn^2}{2}+n-1$.
\end{proof}

In what follows, we examine the upper bound of the stacked-book graph $G_{m,n}$.

\begin{lemma} \label{lem1}
Let $G(i)$ be a subgraph of $G_{m,n}$ induced by $\left\{V_{(i)},V_{(i+\frac{n}{2})}\right\}$. Then for any pair $v \in V_{(i)}$ and $u \in V_{(i+\frac{n}{2})}$, such that $d(u,v) \geq \frac{n}{2}+1$, $|f(v)-f(u)| \geq \frac{n}{2}$.
\end{lemma}

\begin{proof} Let $u=u_{k,i} \in V_{(i)}$ and $v=u_{t,i+\frac{n}{2}} \in V_{(i+\frac{n}{2})}$. Since $d(u,v)>\frac{n}{2}$, then by Remark \ref{reme1}, $k \neq t$. Suppose that neither $u$ nor $v$ is the center vertex of their respective stars $S_{m(i)}$ and $S_{m(i+\frac{n}{2})}$. Then, $d(u,v)=diam(G(i))$. Now, let the radio label on $u$ and $v$ be $f(u)$ and $f(v)$ respectively. Suppose, without loss of generality, that $f(v) > f(u)$. Then $f(v) \geq f(u)+diam(G_{m,n})+1-diam(G(i))$, which implies that
\begin{eqnarray}
f(v) &\geq& f(u)+\frac{n}{2}. \nonumber
\end{eqnarray}
This implies that $f(v)-f(u) \geq \frac{n}{2}$. Similarly, if $f(u) \geq f(v)$, then $f(u)-f(v) \geq \frac{n}{2}$ and thus, the claim follows.
\end{proof}
The following remarks can be confirmed by applying similar methods as in proof of Lemma \ref{lem1}.
\begin{remark}
Suppose that either of $u,v$ in Lemma \ref{lem1}, say $u$, is such that for any $u' \in V_{(i)}$, $uu'\in E(S_{m(i)})$. Then $d(u,u')=\frac{n}{2}+1$ and $|f(u)-f(v)| \geq \frac{n}{2}+1$.
\end{remark}
\begin{remark}
Let $u,u' \in V_{(i)}$. If $d(u,u')=1$, then $|f(u)-f(u')| \geq n+1$ and $|f(u)-f(u')| \geq n$ for $d(u,u')=2$.
\end{remark}

\begin{theorem}\label{thm1}
Let $m > 3$ be odd and $G(i) \subseteq G_{m,n}$, be induced by $\left\{V_{(i)},V_{(i+\frac{n}{2})}\right\}$. then, $rn(G(i)) \leq f(v_1)+mn- \frac{n}{2}+2$, where $v_1$ is the center star $S_{m(1+\frac{n}{2})}$.

\end{theorem}

\begin{proof}
Let $V_{(i)}=\left\{u_{1,i}, u_{2,i}, \cdots, u_{m,i}\right\}$ and $V_{(t)}=\left\{u_{1,t}, u_{2,t}, \cdots, u_{m,t} \right\}$, where $t=i+\frac{n}{2}$. For $r \in [1,m]$, set $u_{r,i} \in V_{(i)}$ as $\alpha_r$ and $u_{r,t} \in V_{(t)}$ as $\beta_r$. From earlier remark, $d(\beta_r, \alpha_r) \in \left\{ \frac{n}{2}+1, \frac{n}{2}+2\right\}$ for $r \neq s$. Now, for every pair $\alpha_s,\beta_r$, where $\alpha_s \in V_{(i)}$, and $\beta_r \in V_{(t)}$, let $r \neq s$ except otherwise stated. Let $\alpha_1$ and $\beta_1$ be the respective centers of the stars $S_{m(i)}$ and $S_{m(t)}$ induced by $V_{(i)}$ and $V_{(t)}$ and let the radio label on $\beta_1$ be $f(\beta_1)$ such that $f(\beta_1)=${\rm{min}}$ \left\{f(\beta_i): 1 \leq i \leq m\right\}$. Since $\beta_1$ is the center of $S_{m(t)}$, then given $\alpha_2 \in V_{(i)}$, $d(\beta_1,\alpha_2)=\frac{n}{2}+1$. Now set $p=diam(G_{m,n})+1-d(\beta_1,\alpha_r)$, $r \neq 1$. Hence, $p=n+2-(\frac{n}{2}+1)=\frac{n}{2}+1$. Suppose that $\alpha_j \in V_{(i)}$ and $\beta_k \in V_{(t)}$, such that $1 \neq j \neq k \neq 1$. Then, $d (\alpha_j, \beta_k)= \frac{n}{2}+2$. So we set $q=diam(G_{m,n})+1-d(\alpha_j,\beta_k)=\frac{n}{2}$. For $f(\beta_1)$ and some $\alpha_2 \in V_{(i)}$, $f(\alpha_2) = f(\beta_1)+p$. Also, for $\beta_3 \in V_{(t)}$, $f(\beta_3) = f(\alpha_2)+q = f(\beta_1)+p+q$ and $f(\alpha_4) = f(\beta_1)+2q+p$. We continue to label the vertices on both $V_{(i)}$ and $V_{(t)}$ alternatively based on the last value attained. Therefore, for $m$ odd,

\begin{eqnarray}
 f(\beta_m)  &=& f(\alpha_{m-1})+ \frac{n}{2} \nonumber\\
 &=& f(\beta_1)+(m-2)q+p. \nonumber
\end{eqnarray}
It can be seen that there does not exist $\alpha_d \in V_{(i)}$, such that $d > m$. So, we reverse the order of labeling, such that for $\beta_m, \alpha_3$, $f(\alpha_3) = f(\beta_m) +q = f(\beta_1)+(m-2)q+2p$. Also, for the pair $\alpha_3, \beta_2$, $f(\beta_2) = f(\beta_1)+(m-2)q+2q+p$. This continues until we reach the pair $\alpha_m$, $\beta_{m-1}$, and obtain
\begin{eqnarray}
f(\alpha_{m-1}) &=& f(\beta_1)+(2m-3)q+p. \nonumber
\end{eqnarray}
Finally, we consider the pair $\beta_{m-1}$ and $\alpha_1$. Since $\alpha_1$ is the center of $S_{(i)}$, then $d(\alpha_1,\beta_{m-1})=\frac{n}{2}+1$ and hence,
\begin{eqnarray}
f(\alpha_1) &=& f(\alpha_{m-1})+p \nonumber \\
&=& f(\beta_1) +(2m-3)q+2p \nonumber \\
 &=& f(\beta_1)+mn-\frac{n}{2}+2. \nonumber
\end{eqnarray}
Hence, $rn(G(i)) \leq f(v_1) +mn-\frac{n}{2} +2 $, where $m$ is odd and $n$ even.
\end{proof}
Next we directly apply Theorem \ref{thm1}.
\begin{lemma} \label{lem2}
Let $\bar{G}(i)$ be induced by $\left\{S_{m(i)}, S_{m(i+\frac{n}{2})}, \gamma_1\right\}$, where $\gamma_1$ is the center of star $S_{m(i+\frac{n}{2}+1)}$, induced by $V_{(i+\frac{n}{2}+1)}$. Then, $f(\gamma_1) \leq f(\beta_1)+mn+3$.
\end{lemma}

\begin{proof}
For $\alpha_1$ and $\beta_1$ centers of stars $S_{(i)}$ and $S_{(i+\frac{n}{2})}$ respectively, let $f(\alpha_1) = f(\beta_1)+mn-\frac{n}{2}+2$, as established in Theorem \ref{thm1}. Then, $d(\alpha_1,\gamma_1)=\frac{n}{2}+1$. Therefore,
\begin{eqnarray}
f(\gamma_1) &=& f(\alpha_1)+p \nonumber \\
&=& f(\beta_1)+mn+3. \nonumber
\end{eqnarray}

\end{proof}
Now, for $\beta_1$, the center of $S_{m(1+\frac{n}{2})}$, induced by $V_{(1+\frac{n}{2})}$. By setting $f(\beta_1)=0$, we establish an upper bound for the radio number of a stacked-book graph $G_{m,n}$ in the next results.
\begin{theorem} \label{thm2}
For $G_{m,n}$, $m$ odd and $n$ even, $rn(G_{m,n}) \leq \frac{mn^2}{2}+n-1$.
\end{theorem}
\begin{proof}
Let $\left\{v_{1(1)},v_{2(1)},v_{3(1)}, \cdots, v_{n(1)}\right\}$ be the set of the respective centers of stars $S_{m(1)}$, $S_{m(2)}$, $S_{m(3)}$, $\cdots$, $S_{m(n)}$ in $G_{m,n}$. Also, suppose that $f(v_{\frac{n}{2}+1(1)})=0$. From the Lemma \ref{lem2} $f(v_{\frac{n}{2}+2(1)})=mn+3$; $f(v_{\frac{n}{2}+3(1)})=2(mn+3)$ and so on. In the end, $f(v_{n(1)})=(\frac{n}{2}-1)(mn+3)$. Also, let $v_{n-\frac{n}{2}(1)}=v_{\frac{n}{2}(1)}$ be the center of $S_{m(\frac{n}{2})} \subset G_{m,n}$ and let $S_{m(\frac{n}{2})}, S_{m(n)}$ induce the graph $G({\frac{n}{2}}) \subset G_{m,n}$. By Theorem \ref{thm1},
\begin{eqnarray}
rn\left(G\left(\frac{n}{2}\right)\right) &\leq& f(v_{n(1)})+mn-\frac{n}{2}+2 \nonumber \\
&\leq& \frac{mn^2}{2}+n-1 \nonumber.
\end{eqnarray}
\end{proof}
\begin{theorem} \label{thm03}
Let $m,n$ be even. Then $rn(G_{m,n}) \leq \frac{mn^2}{2}+n-1$.
\end{theorem}
\begin{proof}
The proof follows similar argument and technique as in Theorem \ref{thm1}, Lemma \ref{lem2} and Theorem \ref{thm2}.
\end{proof}


\begin{center}
\pgfdeclarelayer{nodelayer}
\pgfdeclarelayer{edgelayer}
\pgfsetlayers{nodelayer,edgelayer}
\begin{tikzpicture}
	\begin{pgfonlayer}{nodelayer}

		\node [minimum size=0cm,draw, circle] (0) at (1.5,0) {\tiny 16};
		\node [minimum size=0cm,draw, circle] (1) at (3,0) {\tiny 43};
		\node [minimum size=0cm,draw, circle] (2) at (4.5,0) {\tiny 70};
		\node [minimum size=0cm,draw, circle] (3) at (6,0) {\tiny 7};
	  \node [minimum size=0cm,draw, circle] (4) at (7.5,0) {\tiny 34};
		\node [minimum size=0cm,draw, circle] (4a) at (9.0,0) {\tiny 61};

		\node [minimum size=0cm,draw, circle] (5) at (0.5,1) {\tiny 10};
		\node [minimum size=0cm,draw, circle] (6) at (2.0,1) {\tiny 37};
		\node [minimum size=0cm,draw, circle] (7) at (3.5,1) {\tiny 64};
		\node [minimum size=0cm,draw, circle] (8) at (5,1) {\tiny 19};
		\node [minimum size=0cm,draw, circle] (9) at (6.5,1) {\tiny 46};
		\node [minimum size=0cm,draw, circle] (9a) at (8,1) {\tiny 73};
		
		\node [minimum size=0cm,draw, circle] (10) at (1,2) {\tiny 23};
		\node [minimum size=0cm,draw, circle] (11) at (2.5,2) {\tiny 50};
		\node [minimum size=0cm,draw, circle] (12) at (4,2) {\tiny 77};
	  \node [minimum size=0cm,draw, circle] (13) at (5.5,2) {\tiny 0};
		\node [minimum size=0cm,draw, circle] (14) at (7,2) {\tiny 27};
		\node [minimum size=0cm,draw, circle] (14a) at (8.5,2) {\tiny 54};

		\node [minimum size=0cm,draw, circle] (15) at (1,3.5) {\tiny 4};
		\node [minimum size=0cm,draw, circle] (16) at (2.5,3.5) {\tiny  31};
		\node [minimum size=0cm,draw, circle] (17) at (4,3.5) {\tiny 58};
		\node [minimum size=0cm,draw, circle] (18) at (5.5,3.5) {\tiny 13};
		\node [minimum size=0cm,draw, circle] (19) at (7,3.5) {\tiny 40};
		\node [minimum size=0cm,draw, circle] (19a) at (8.5,3.5) {\tiny 67};
		
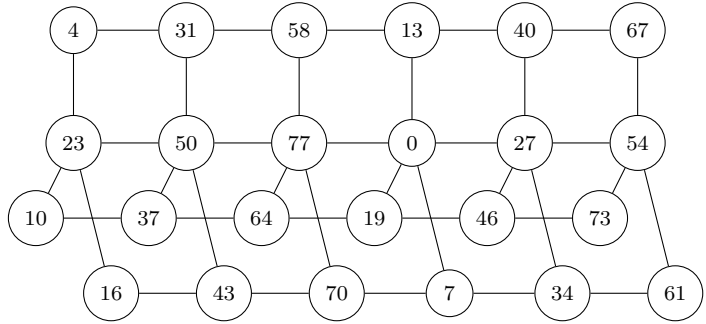
	  \node [minimum size=0] (20) at (5,-1) {\small Figure 1. A $G_{4,6}$ graph with $rn(G_{4,6}) \leq 77$};
		
	\end{pgfonlayer}
	\begin{pgfonlayer}{edgelayer}
		\draw [thin=1.00] (0) to (1);
		\draw [thin=1.00] (1) to (2);
		\draw [thin=1.00] (2) to (3);
		\draw [thin=1.00] (3) to (4);
		\draw [thin=1.00] (4) to (4a);
		
		\draw [thin=1.00] (5) to (6);
		\draw [thin=1.00] (6) to (7);
    \draw [thin=1.00] (7) to (8);
    \draw [thin=1.00] (8) to (9);
		\draw [thin=1.00] (9) to (9a);

		\draw [thin=1.00] (10) to (11);
		\draw [thin=1.00] (11) to (12);
		\draw [thin=1.00] (12) to (13);
		\draw [thin=1.00] (13) to (14);
		\draw [thin=1.00] (14) to (14a);

		\draw [thin=1.00] (15) to (16);
    \draw [thin=1.00] (16) to (17);
    \draw [thin=1.00] (17) to (18);
		\draw [thin=1.00] (18) to (19);
		\draw [thin=1.00] (19) to (19a);
		
		\draw [thin=1.00] (10) to (5);
		\draw [thin=1.00] (10) to (15);
		\draw [thin=1.00] (10) to (0);
		
		\draw [thin=1.00] (11) to (16);
    \draw [thin=1.00] (11) to (6);
    \draw [thin=1.00] (11) to (1);
		
		\draw [thin=1.00] (12) to (17);
		\draw [thin=1.00] (12) to (7);
		\draw [thin=1.00] (12) to (2);
		
		\draw [thin=1.00] (13) to (18);
		\draw [thin=1.00] (13) to (8);
    \draw [thin=1.00] (13) to (3);
		
    \draw [thin=1.00] (14) to (19);
		\draw [thin=1.00] (14) to (9);
    \draw [thin=1.00] (14) to (4);
		
		\draw [thin=1.00] (14a) to (4a);
		\draw [thin=1.00] (14a) to (9a);
    \draw [thin=1.00] (14a) to (19a);

	\end{pgfonlayer}
\end{tikzpicture}

\end{center}

Theorems \ref{thm11}, \ref{thm2}, \ref{thm03}
establish the radio number of $G_{m,n}$, where $m \geq 4$ and $n$ is even, as recapped in the next theorem.
\begin{theorem}\label{thm0001}
Let $G_{m,n}$ be a stacked-book graph with $m \geq 4$ and $n$ even, then, $rn(G_{m,n})=\frac{mn^2}{2}+n-1$.
\end{theorem}
Next we consider the case where $m=3$. First we present a result that is equivalent to Theorem \ref{thm1} with respect to $m=3$.

\begin{theorem} \label{thm3}
Let $G_{3,n}$ be a stacked-book graph, where $n$ is even. Suppose that the pair $\left\{S_{3(i)}, S_{3(i+\frac{n}{2})}\right\}$ form a subgraph $G(i)$ of $G_{3,n}$. Then, $rn(G(i)) \leq f(u_1)+\frac{5n}{2}+3$, where $u_1$ is the center vertex of $S_{3(i+\frac{n}{2})}$.
\end{theorem}
\begin{proof}
Let $V_{(i)}=\left\{v_1,v_2,v_3\right\}$ and $V_{(i+\frac{n}{2})}=\left\{u_1,u_2,u_3\right\}$ where $V_{(i)}$ and $V_{(i+\frac{n}{2})}$ are vertex sets of stars $S_{3(i)}$ and $S_{3(i+\frac{n}{2})}$ in $G_{3,n}$ respectively. Also, let $v_1$ and $u_i$ be the respective center vertices of $S_{3(i)}$ and $S_{3(i+\frac{n}{2})}$. From earlier remark, $d(v_1,u_1)=\frac{n}{2}+1$. Suppose that $f(u_1)$, the radio label of $u_1$ is the smallest possible radio label on $G(i)$, then,
\begin{eqnarray}
f(v_2) &=& f(v_1)+diam(G_{3,n})+1-d(v_1,u_1) \nonumber \\
&=& f(u_i)+ \frac{n}{2}+1. \nonumber
\end{eqnarray}
For $v_2,u_3$ $d(v_2,u_3)=\frac{n}{2}+2$,
\begin{eqnarray}
f(u_3) &=&f(u_1)+n+1 \nonumber
\end{eqnarray}
For $u_3,v_1$, $d(u_3,v_1)= \frac{n}{2}+1$,
\begin{eqnarray}
f(v_1) &=& f(u_1)+\frac{3n}{2}+2. \nonumber
\end{eqnarray}
For $v_1,u_2$, $d(v_1,u_2)=\frac{n}{2}+1$ and thus,
\begin{eqnarray}
f(u_2) &=& f(u_1)+\frac{3n}{2}+2+n+2-\left(\frac{n}{2}+1\right) \nonumber \\
&=& f(u_1)+2n+3. \nonumber
\end{eqnarray}
And finally, for the pair $v_3,u_2$, $d(v_3,u_3)=\frac{n}{2}+2$ and
\begin{eqnarray}
f(u_3) &=& f(u_1)+\frac{5n}{2}+3. \nonumber
\end{eqnarray}
Hence, $rn(G(i)) \leq f(u_1)+\frac{5n}{2}+3$.

\end{proof}
Next, we obtain the following result.
\begin{lemma}
Let $\kappa_1$ be the center of star $S_{3(i+\frac{n}{2})+1} \subseteq G_{3,n}$ and let $\bar{H}(1)$ be a subgraph of $G_{(3,m)}$ induced by $\left\{S_{3(i)}, S_{3(i+\frac{n}{2}), \kappa_1}\right\}$. Then $f(\kappa_1) \leq 3n+1$.
\end{lemma}
\begin{proof}
The vertex with the maximum value of radio label in Theorem \ref{thm3} is $u_3$. Let us adopt this, with $f(u_3)=f(u_1)+\frac{5n}{2}+3$. Now, $d(u_3,\kappa_1)=\frac{n}{2}+2$. Therefore,
\begin{eqnarray}
f(\kappa_1) &=& f(u_1)+3n+3. \nonumber
\end{eqnarray}
\end{proof}
In the final result here, we set $f(u_1)=0,$ for $u_i$, the center of star $S_{3(1+\frac{n}{2})}$.
\begin{theorem} \label{thm4}
Let $n$ be an even positive integer. Then, $rn(G_{3,m}) \leq \frac{3n^2}{2}+n$.
\end{theorem}
\begin{proof}
Proof follows similar technique adopted in Theorem \ref{thm03}.
\end{proof}

Figure 2 is a radio numbering for $G_{3,6}$. It shows that $rn(G_{3,6})$ is not more than $60$.

\begin{center}
\pgfdeclarelayer{nodelayer}
\pgfdeclarelayer{edgelayer}
\pgfsetlayers{nodelayer,edgelayer}
\begin{tikzpicture}
	\begin{pgfonlayer}{nodelayer}

		\node [minimum size=0cm,draw, circle] (0) at (1,0) {\tiny 18};
		\node [minimum size=0cm,draw, circle] (1) at (2.5,0) {\tiny 39};
		\node [minimum size=0cm,draw, circle] (2) at (4,0) {\tiny 60};
		\node [minimum size=0cm,draw, circle] (3) at (5.5,0) {\tiny 7};
	  \node [minimum size=0cm,draw, circle] (4) at (7,0) {\tiny 28};
		\node [minimum size=0cm,draw, circle] (4a) at (8.5,0) {\tiny 59};

		\node [minimum size=0cm,draw, circle] (10) at (1,2) {\tiny 11};
		\node [minimum size=0cm,draw, circle] (11) at (2.5,2) {\tiny 32};
		\node [minimum size=0cm,draw, circle] (12) at (4,2) {\tiny 53};
	  \node [minimum size=0cm,draw, circle] (13) at (5.5,2) {\tiny 0};
		\node [minimum size=0cm,draw, circle] (14) at (7,2) {\tiny 21};
		\node [minimum size=0cm,draw, circle] (14a) at (8.5,2) {\tiny 42};

		\node [minimum size=0cm,draw, circle] (15) at (1,4) {\tiny 4};
		\node [minimum size=0cm,draw, circle] (16) at (2.5,4) {\tiny  25};
		\node [minimum size=0cm,draw, circle] (17) at (4,4) {\tiny 46};
		\node [minimum size=0cm,draw, circle] (18) at (5.5,4) {\tiny 15};
		\node [minimum size=0cm,draw, circle] (19) at (7,4) {\tiny 36};
		\node [minimum size=0cm,draw, circle] (19a) at (8.5,4) {\tiny 57};
		
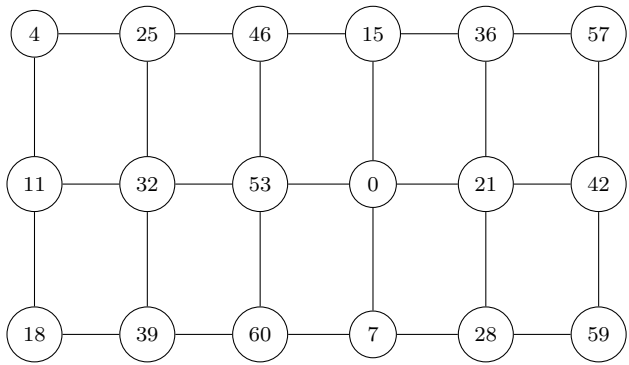
	  \node [minimum size=0] (20) at (5,-1) {\small Figure 2. A $G_{3,6}$ graph with $rn(G_{3,6}) \leq 60$};
		
	\end{pgfonlayer}
	\begin{pgfonlayer}{edgelayer}
		\draw [thin=1.00] (0) to (1);
		\draw [thin=1.00] (1) to (2);
		\draw [thin=1.00] (2) to (3);
		\draw [thin=1.00] (3) to (4);
		\draw [thin=1.00] (4) to (4a);

		\draw [thin=1.00] (10) to (11);
		\draw [thin=1.00] (11) to (12);
		\draw [thin=1.00] (12) to (13);
		\draw [thin=1.00] (13) to (14);
		\draw [thin=1.00] (14) to (14a);

		\draw [thin=1.00] (15) to (16);
    \draw [thin=1.00] (16) to (17);
    \draw [thin=1.00] (17) to (18);
		\draw [thin=1.00] (18) to (19);
		\draw [thin=1.00] (19) to (19a);
		
		\draw [thin=1.00] (10) to (15);
		\draw [thin=1.00] (10) to (0);
		
		\draw [thin=1.00] (11) to (16);
    \draw [thin=1.00] (11) to (1);
		
		\draw [thin=1.00] (12) to (17);
		\draw [thin=1.00] (12) to (2);
		
		\draw [thin=1.00] (13) to (18);
    \draw [thin=1.00] (13) to (3);
		
    \draw [thin=1.00] (14) to (19);
    \draw [thin=1.00] (14) to (4);
		
		\draw [thin=1.00] (14a) to (4a);
    \draw [thin=1.00] (14a) to (19a);

	\end{pgfonlayer}
\end{tikzpicture}

\end{center}

\section{Conclusion}

It is noteworthy to look at some of the results in \cite{J1}. A $G_{3,n}$ is a $3 \times n$ grid. By \cite{J1}, it is seen that $rn(G_{3,6}) = 59$, which is better than the result in Figure 2 above by 1. But this is still a considerable improvement compared with a upper bound of $109$ suggested in \cite{AA1}. In establishing the upper bound for $G_{3,n}$, it is observed that the number of the pair $u,v \in V(G_{3,n})$ for which $d(u,v) = \frac{diam(G_{3,n})+1}{2}$ is more than the case where $d(u,v)=\frac{n}{2}$ in each of the segments of radio labeling of the stacked-graph. However, the reverse proves to be the case in $G_{m,n}$, $m \geq 4$.



\begin{thebibliography}{99}
\bibitem[1]{AA1} D. O. Ajayi and T. C. Adefokun {\it {On bounds of radio number of certain product graphs}}, J. Nigerian Math. Soc. 37(2) (2018) 71-78.

\bibitem[2]{BD1} D. Bantva, S Vaidya and S. Zhou {\it{Radio numbers of trees}}, Electron. Notes in Discrete Math. 48 (2015) 135-141.


\bibitem[3]{BD2} D. Bantva  {\it{Radio numbers of middle graph of paths}}, Electron. Notes in Discrete Math. 66 (2017) 93-100.



\bibitem[4]{CEHZ1} G. Chartrand, D. Erwin, F Harary and P. Zhang, {\it {Radio labelings of graphs}}, Bull. Inst. Combin. Appl. 33 (2001) 77-85.
	
		\bibitem[5]{CEZ1} G. Chartrand, D. Erwin and P. Zhang, {\it{A graph labeling problem suggested}} by FM Channel Restrictions, Bull. Inst. Combin. Appl. 43 (2005) 43-57.
		\bibitem[6]{Hale} W. K. Hale, {\it{Frequency assignment theory and applications}}, Proc. IEEE, 68 (1980) 1497 - 1514

		\bibitem[7]{J1} T-S Jiang, {\it{The radio number of grid graphs}}, arXiv:1401.658v1. 2014.
		
		\bibitem[8]{LX2} D.D.-F. Liu and M. Xie {\it{Radio number for square paths}}, Ars Combin. 90 (2009) 307-319.

\bibitem[9]{LX1} D.D.-F. Liu and M. Xie {\it{Radio number for square cycles}}, Congr. Numer. 169 (2004) 105-125.
		
		\bibitem[10]{LZ1} D. Liu and X. Zhu, {\it{Multilevel distance labelings for paths and cycles}}, SIAM J. Discrete Math. 19 (2005) 610-621.
		
		\bibitem[11]{NSS1} A. Naseem, K. Shabbir and H. Shaker {\it{The radio number of edge-joint graphs}}, ARS Comb. 139 (2018) 337-351.
		
		
		
		\bibitem[12]{SP1} L. Saha and P. Panigrahi, {\it{On the radio numbers of toroidal grid}}, Aust. Jour. Combin. 55 (2013) 273-288.
		



\end{thebibliography}
\end{document}